\documentclass[12pt, reqno, twoside, letterpaper]{amsart}

\usepackage{amsmath,amssymb,amsbsy,amsfonts,amsthm,latexsym,
	amsopn,amstext,amsxtra,euscript,amscd,stmaryrd,mathrsfs,
	cite,array}

\numberwithin{equation}{section}

\usepackage{color}
\usepackage[usenames,dvipsnames,svgnames,table]{xcolor}

\def\eps{\varepsilon}

\def\fl#1{\left\lfloor#1\right\rfloor}

\def\({\left(}
\def\){\right)}

\newcommand{\e}{\ensuremath{\mathbf{e}}}

\newcommand{\cL}{\ensuremath{\mathcal{L}}}

\newcommand{\cT}{\ensuremath{\mathcal{T}}}

\newcommand{\NN}{\ensuremath{\mathbb{N}}}
\newcommand{\PP}{\ensuremath{\mathbb{P}}}

\newcommand{\RR}{\ensuremath{\mathbb{R}}}
\newcommand{\ZZ}{\ensuremath{\mathbb{Z}}}

\usepackage[margin=2.5cm]{geometry}

\usepackage{amsthm}
\newtheoremstyle{customthm}
{1em}                    
{1em}                    
{\itshape}               
{}                       
{\scshape}               
{.}                      
{5pt plus 1pt minus 1pt} 
{}                       

\newtheoremstyle{customrem}
{1em}                    
{1em}                    
{}                       
{}                       
{\scshape}               
{.}                      
{5pt plus 1pt minus 1pt} 
{}                       

\theoremstyle{customthm}

\newtheorem{X}{X}[section]
  
\newtheorem{lemma}[X]{Lemma}
\newtheorem{proposition}[X]{Proposition}

\newtheorem{Y}{Y}
\newtheorem{theorem}[Y]{Theorem}  

\newtheorem{Z}{Z}
\newtheorem{corollary}[Z]{Corollary}

\theoremstyle{customrem}

\usepackage{etoolbox}
\AtEndEnvironment{remark}{\null\hfill\qedsymbol}

\AtEndEnvironment{definition}{\null\hfill\qedsymbol}

\renewcommand{\le}{\ensuremath{\leqslant}}
\renewcommand{\ge}{\ensuremath{\geqslant}}

\renewcommand{\pod}[1]{\mathchoice
	{\allowbreak \if@display \mkern 5mu\else \mkern 5mu\fi (#1)}
	{\allowbreak \if@display \mkern 5mu\else \mkern 5mu\fi (#1)}
	{\mkern4mu(#1)}
	{\mkern4mu(#1)}
}

\newcommand*{\defeq}{\mathrel{\vcenter{\baselineskip0.5ex \lineskiplimit0pt
			\hbox{\scriptsize.}\hbox{\scriptsize.}}}%
	=}

\DeclareSymbolFont{EUEX}{U}{euex}{m}{n}

\DeclareSymbolFont{euexlargesymbols}{U}{euex}{m}{n}
\DeclareMathSymbol{\intop}{\mathop}{euexlargesymbols}{"52}
\def\int{\intop\nolimits}

\DeclareSymbolFont{euexsymbols}     {U}{euex}{m}{n}
\DeclareMathSymbol{\smallint}{\mathop}{euexsymbols}{"52}

\allowdisplaybreaks

\title[Additive problems on $\lfloor p^c \rfloor$]
{Additive problems on $\lfloor p^c \rfloor$}

\author[Lingyu Guo]{Lingyu Guo}

\address{School of Mathematics and Statistics, Xi'an Jiaotong University, Xi'an,Shaanxi,China.}

\email{guo.lingyu@foxmail.com}

\author[Victor Zhenyu Guo]{Victor Zhenyu Guo}

\address{School of Mathematics and Statistics, Xi'an Jiaotong University, Xi'an,Shaanxi,China.}

\email{vzguo@foxmail.com; guozyv@xjtu.edu.cn}

\author[Li Lu]{Li Lu}

\address{School of Mathematics and Statistics, Xi'an Jiaotong University, Xi'an,Shaanxi,China.}

\email{lilu\_math@foxmail.com}

\date{\today}

\begin{document}
	
\begin{abstract}

The sequence
$$
\mathbb{P}^{(c)}=(\lfloor p^c \rfloor)_{p\in \mathbb{P}}\quad (c>0,c\notin \mathbb{N}),
$$
is an important subsequence of the well-known Piatetski-Shapiro sequence, where $\PP$ is the set of prime numbers and $\lfloor \cdot \rfloor$ is the floor function. We prove that for all $c \in (0, 13/15)$, any large enough integer $N$ can be represented as
	$$
	N=\lfloor p^c\rfloor+q,
	$$
where $p$ and $q$ are primes. We also prove the result holds for almost all fixed positive $c \in \mathbb{R}\setminus\mathbb{Z}$. Moreover, we investigate shifted primes in this sequence, obtaining an asymptotic formula for all $c \in (0, 13/15)$ and an almost-all result for fixed positive $c \in \mathbb{R}\setminus\mathbb{Z}$.

\end{abstract}
	
\maketitle
	
\begin{quote}
\end{quote}

\begin{quote}
\end{quote}
	
\section{Introduction}
\subsection{Motivation}
We are interested in the following sequence which is 
$$
\PP^{(c)}=(\lfloor p^c \rfloor)_{p\in \mathbb{P}}\quad (c>0,c\notin \NN),
$$
where $\PP=\{2,3,5,\cdots\}$ is the set of prime numbers and $\lfloor \cdot \rfloor$ is the floor function. Note that the set $\mathbb{P}^{(c)}$ is highly related to the Piatetski-Shapiro sequences of the form
$$
\NN^{(c)}=(\lfloor n^c \rfloor)_{n\in \mathbb{N}}\quad (c>0,c\notin \NN),
$$
which is named in honor of Piatetski-Shapiro, who showed (cf.\cite{PS}) that for any fixed $c \in (1, 12/11)$ there are infinitely many primes in $\NN^{(c)}$. The admissible range of $c$ for this result has been extended many times over the years, and currently it is known to hold for all $c \in (1, 243/205)$ thanks to Rivat and Wu \cite{J.Wu}.

Many authors have studied arithmetic properties of Piatetski-Shapiro sequences (see Baker et al.\cite{Baker} and the references contained therein), and it is natural to ask whether certain properties also hold on special subsequences of the Piatetski-Shapiro sequences. The set $\mathbb{P}^{(c)}$ is one of the most interesting subsequences; however, up to now very little has been established about the prime numbers in $\PP^{(c)}$ for fixed $c > 1$. Balog \cite{Balog} has shown that $c\in (0, 5/6)$, the counting function
\begin{equation}
\label{eq:pc}
\Pi_{c, 0}(x) =\sum_{\substack{p\leqslant x \\ \lfloor p^c \rfloor \in \PP}}1\sim \frac{x}{c\log ^2 x}.
\end{equation}
Moreover, Balog also showed that for almost all $c > 0$, there are infinitely many primes in $\PP^{(c)}$. For more results on $\mathbb{P}^{(c)}$, we list as follows.

\begin{itemize}
	\item Square-free numbers and Cube-free numbers: \cite{CZ2008}, \cite{ZL2017};
	\item Almost primes: \cite{BGS}, \cite{XLZ2024};
	\item Exceptional sets on $N = \fl{p_1^c} + \fl{p_2^c}$: \cite{Baker2024}, \cite{Lap1999}, \cite{Zhu2017}. 
\end{itemize}

In this article, we analyze some additive problems on the sequence $\PP^{(c)}$ related to linear representations of prime numbers. 

\subsection{Main results}
Let $p,q$ denote primes throughout the paper, $c\in \RR$, $a$ is an integer and $N$ is an integer large enough. Define
\begin{align*}
&\Pi_{c,a}(x)=\#\{p\leqslant x: \lfloor p^c\rfloor = q+a, q\in \PP\},\\
&\Upsilon_c(N)=\#\{m \leqslant N: N = m + q, m = \lfloor p^c\rfloor, q\in \PP\}.
\end{align*}

We prove the following result related to shifted primes.  

\begin{theorem} For all $c \in (0, 13/15)$ and any fixed integer $a\geqslant 0$, as $x\rightarrow \infty$ we have
	\label{thm:main_1}
	$$
	\Pi_{c,a}(x)=\frac{x}{c\log ^2 x}+ O(\frac{x}{\log ^3x}).
	$$
\end{theorem}
If we take $a=0$, we have the following corollary which improves the result \eqref{eq:pc} by Balog in \cite{Balog}.

\begin{corollary}  For all $c \in (0, 13/15)$, as $x\rightarrow \infty$ we have
	\label{corollary:1}
	$$
	\Pi_{c,0}(x)=\sum_{\substack{p\leqslant x \\ \lfloor p^c \rfloor \in \PP}}1 = \frac{x}{c\log ^2 x}+ O(\frac{x}{\log ^3x}).
	$$
\end{corollary}

We also think about the following additive problem and prove that: 

\begin{theorem} For all $c \in (0, 13/15)$, any large enough integer $N$ can be represented as
	$$
	N=\lfloor p^c\rfloor+q.
	$$
	Moreover, we have
	\label{thm:main_2}
	$$
	\Upsilon_c(N)\gg \frac{N^{1/c}}{\log ^2N}.
	$$
\end{theorem}

Our method is highly related to the behavior of zeros of the Riemann $\zeta$-function. The number $13/15$ is restricted by the upper bound of the zero density estimate $N(\sigma,T)$ due to Guth and Maynard \cite{Maynard} and the zero free region of the Riemann zeta function, where
$$
N(\sigma,T)=\sum_{\substack{\rho\\ \beta>\sigma,\ |\gamma|\leqslant T}}1,
$$
with $\rho =\beta +i\gamma$ being the non-trivial zeros of the Riemann $\zeta$-function. Without another significant improvement on the distribution of the zeros, it is not likely to recover the bound of $c$ by our methods in this paper. 

Hence we consider the problems in an average sense by proving an almost-all result for $c$. For the shifted prime problem, we prove the following result.

\begin{theorem}
	\label{thm:main_3}
 For almost all fixed positive $c \in \RR\setminus\ZZ$ and any fixed integer $a\geqslant 1$, we have (in the sense of Lebesgue measure)
	$$
	\lim\limits_{x\rightarrow\infty}\sup \frac{c \Pi_{c,a}(x)}{x/\log ^2 x}\geqslant 1.
	$$
\end{theorem}

We also obtain an almost-all result for the other additive problem. 

\begin{theorem}	
\label{thm:main_4}
For almost all fixed positive $c \in \RR\setminus\ZZ$ and large enough integer $N$, we have (in the sense of Lebesgue measure)
	$$
	\lim\limits_{N\rightarrow\infty}\sup \frac{\Upsilon_c(N)}{cN^{1/c}/\log ^2 N}\geqslant 1.
	$$
\end{theorem}

We remark that the density of $\fl{p^c}$ is lower than the density of prime numbers when $c > 1$. Hence Theorem \ref{thm:main_4} can be treated as an “almost-all” approach to the binary Goldbach problem.  

\section{Preliminaries}
\subsection{Notation}
Throughout the paper, we use the symbols $O,\ll,\gg$ and $\asymp$ along with their standard meanings; any constants or functions implied by these symbols may depend on $c$ and (where obvious) on the parameters $\varepsilon$. We use the notation $m \simeq M$ as an abbreviation for $M < m \leqslant M+ M/\log M$  and $m\sim M$ for $M<m\leqslant 2M$. The function $\mathcal{L}(S)$ denotes the Lebesgue measure of the set $S$.\par
The letters $p,q$ always denote prime numbers. As usual, $\pi(x)$ denotes the number of primes $\leqslant x$, $\Lambda(\cdot)$ is the von Mangoldt function, $\psi(x)=\sum_{n\leqslant x}\Lambda(n)$, $\chi$ is a Dirichlet character and $\rho =\beta +i\gamma$ is a non-trivial zero of the Riemann $\zeta$-function. \par
Let $\e(t)= e^{2\pi it}$ for all $t\in\RR$. $\|t\|$ denotes the distance from the real number $t$ to the nearest integer; that is,
$$
\|t\|=\min_{n\in \ZZ}|t-n|\quad (t\in \RR).
$$
We denote by $\lfloor t\rfloor$ and $\{t\}$ the greatest integer $\leqslant t$ and the fractional part of $t$. Let $\Psi(t)=t-\lfloor t\rfloor- 1/2$.
\subsection{The explicit formula of $\psi(x)$}
The explicit formula of $\psi(x) = \sum_{n \leqslant x} \Lambda(n)$ plays an important role in our proof.  
\begin{lemma}
	\label{lemma:explicit-formula}
	Let $c>1$ be a constant. Suppose that $x\leqslant c$, $T\leqslant 2$, and let $\langle x\rangle$ denote the distance from $x$ to the nearest prime power, other than $x$ itself. Then
	\begin{equation}
	\label{eq:psi0}
	\psi_0(x)=x-\sum_{\substack{\rho \\ |\gamma|\leqslant T}}\frac{x^\rho}{\rho}-\log 2\pi-\frac{1}{2}\log(1-\frac{1}{x^2})+R(x,T),
	\end{equation}
	where
	$$
	R(x,T)\ll (\log x)\min(1,\frac{x}{T\langle x\rangle})+\frac{x}{T}(\log xT)^2
	$$
	and
	$$
	\psi_0(x)=\frac{1}{2}(\psi(x^+)+\psi(x^-)).
	$$
\begin{proof}
See \cite[Theorem~12.5]{Montgomery}.
\end{proof}
\end{lemma}

\subsection{The bound for $N(\sigma,T)$ and primes in short intervals}
\begin{lemma}
\label{lemma:Maynard}
Let $N(\sigma,T)$ denote the number of zeros $\rho $ of $\zeta(s)$ with $\Re(\rho)\geqslant \sigma$ and $\Im(\rho)\leqslant T$. It follows that
$$
N(\sigma,T)\ll T^{30(1-\sigma)/13+o(1)}.
$$
\begin{proof}
See \cite[Theorem~1.2]{Maynard}.
\end{proof}
\end{lemma}

\subsection{A mean value theorem}
\begin{lemma}
\label{lem:mean_value}
Let 
$$
S(s,\chi)=\sum_{n=1}^{N}a_n\chi(n)n^{-s},
$$
and let $\mathcal{A}_\chi$ be a finite set of complex numbers $s=\sigma+it$. Let $T_0,T,\sigma_0,\delta$ be real numbers such that for each $\chi$,
$$
T_0+\frac{\delta}{2}\leqslant t\leqslant T_0+T-\frac{\delta}{2},\quad \sigma\geqslant \sigma_0 \quad \text{and} \quad |t-t'|\geqslant \delta 
$$
for all $s,s'\in\mathcal{A}_\chi$ and $s\neq s'$. Then 
$$
\sum_{\chi}\sum_{s\in \mathcal{A}_\chi}|S(s,\chi)|^2\ll (qT+N)(\delta^{-1}+\log N)\sum_{n=1}^{N}|a_n|^2n^{-2\sigma_0}\bigg( 1+\log \frac{\log 2N}{\log 2n} \bigg).
$$ 
\end{lemma}

\begin{proof}
See \cite[Theorem~7.6]{Montgomery2}.
\end{proof}

\subsection{The zero-free region}

We also need the classical result of the zero-free region of the Riemann $\zeta$-function; see \cite[Corollary~11.4]{Montgomery2}.
\begin{lemma}
\label{zero-free}
There is an absolute constant $c>0$ such that if $\zeta(\rho)=0, \rho=\beta+i\gamma, \tau=|\gamma|+2$, then
$$
\beta \leqslant 1-c(\log \tau)^{-2/3}(\log\log\tau)^{-1/3}.
$$
\end{lemma}

\subsection{Technical lemmas }
We also need the following well known approximation of Vaaler~\cite{Vaal} and a useful lemma for exponential sums.

\begin{lemma}
	\label{lem:Vaaler}
	For any $H\ge 1$ there are numbers $a_h,b_h$ such that
	$$
	\bigg|\Psi(t)-\sum_{0<|h|\le H}a_h\,\e(th)\bigg|
	\le\sum_{|h|\le H}b_h\,\e(th),\qquad
	a_h\ll\frac{1}{|h|}\,,\qquad b_h\ll\frac{1}{H}\,.
	$$
\end{lemma}

\begin{lemma}
	\label{exponential_sum}
Let $f(x)$ be real and have continuous derivatives up to the $k$-th order, where $k\leqslant 4$. Let $\lambda_k\leqslant f^{(k)}(x)\leqslant h\lambda_k$. Let $b-a\geqslant 1$, $K=2^{k-1}$. Then
$$
\sum_{a<n\leqslant b}\e(f(n))\ll (b-a)\lambda_k^{1/(2K-2)}+(b-a)^{1-2/K}\lambda_k^{-1/(2K-2)},
$$
where the constants implied are independent of $k$.
\end{lemma}
\begin{proof}
See \cite[Theorem~5.13]{Titc}.
\end{proof}

\section{Proof of theorem~\ref{thm:main_1}}

\subsection{An initial construction}
We start with the set-up by Balog~\cite{Balog}. Assuming $0<c<1$ it is trivial that
$$
\lfloor p^c \rfloor =q+a   \Leftrightarrow (q+a)^{1/c} \leqslant p <(q+a+1)^{1/c}.
$$
Hence the problem changes into a question of the sum of primes in short intervals which is 
\begin{align}
\label{eq:pic}
\Pi_{c,a}(x) &=\sum_{q+a\leqslant x^c} \sum_{\substack{ (q+a)^{1/c} \leqslant p <(q+a+1)^{1/c}}}1 \nonumber \\
&= \sum_{q\leqslant x^c} \big( \pi((q+a+1)^{1/c})-\pi((q+a)^{1/c}) \big)+O(1).
\end{align}
Next, we change $\pi(x)$ into $\psi(x)$ by a partial summation and use Lemma~\ref{lemma:explicit-formula}. Taking $T=x^c\log^5x$, it gives that
$$
R(x,T)\ll \frac{x\log ^2x}{T}
$$
Then we simplify~\eqref{eq:psi0} to
\begin{equation}
\label{eq:psi}
\psi(x)=x-\sum_{\substack{\rho \\ |\gamma|\leqslant T}}\frac{x^\rho}{\rho}+O(\frac{x\log ^2x}{T}).
\end{equation}
We replace $q+a$ by $r$. To estimate \eqref{eq:pic}, it is sufficient to compute
\begin{align*}
\sum_{q\leqslant x^c} \bigg( \psi((r+1)^{1/c})-\psi(r^{1/c}) \bigg) = S_1 - S_2 + O(S_3),
\end{align*}
where
\begin{align*}
	&S_1=\sum_{\substack{r\leqslant x^c \\ r-a \in\PP}} \bigg( (r+1)^{1/c}-r^{1/c} \bigg),\\
	&S_2=\sum_{\substack{r\leqslant x^c \\ r-a \in\PP}}  \sum_{\substack{\rho \\ |\gamma| \leqslant T}}\frac{(r+1)^{\rho/c} - r^{\rho/c}}{\rho} ,\\
	&S_3= \sum_{\substack{r\leqslant x^c \\ r-a \in\PP}} \frac{r\log^2 r}{T}.
\end{align*}
according to \eqref{eq:psi}.

\subsection{The estimates of $S_1$ and $S_3$}
Note that $S_1$ is the main term while $S_2$ and $S_3$ are both error terms. Since $r-a\in\PP$, $r$ is a shift of the prime $q$. By a partial summation and , we obtain that
\begin{align*}
\nonumber S_1 &= \int_{2}^{x^c} \bigg( (u+1)^{1/c}-u^{1/c} \bigg)\,d\pi(u+a) \\
&= \frac{x}{c^2\log x} +O\big(\frac{x}{\log^2x}\big) - \frac{1}{c}\big(\frac{1}{c}-1\big)\int_{2}^{x^c}u^{-2+1/c}\big(1+O(u^{-1})\big) \pi(u+a) \, du.
\end{align*}
According to prime number theorem, we have
$$
\int_{2}^{x^c}u^{-2+1/c} \pi(u+a) \, du = \int_{2}^{x^c}u^{-2+1/c} \frac{u}{\log u}\big(1+O(\log^{-1} u)\big) \, du.
$$
Noting that
$$
\frac{1}{c}\big(\frac{1}{c}-1\big)\int_{2}^{x^c}u^{-2+1/c} \frac{u}{\log u} \, du = \big(\frac{1}{c}-1\big) \int_{2}^{x^c}\frac{du^{1/c}}{\log u} = \frac{1}{c}\big(\frac{1}{c}-1\big) \int^{x}_{2^{1/c}} \frac{dt}{\log t},
$$
it follows that
\begin{align}
	\label{eq:bdS1}
S_1 &= \frac{x}{c^2\log x} - \frac{1}{c}\big(\frac{1}{c}-1\big) \frac{x}{\log x} + O\big(\frac{x}{\log^2x}\big) = \frac{x}{c\log x} + O\big(\frac{x}{\log^2x}\big).
\end{align}
Similarly by a partial summation and prime number theorem, we get that 
\begin{equation}
\label{eq:bdS3}
S_3 \ll \frac{x}{\log ^2x}.
\end{equation}
 
\subsection{The estimate of $S_2$ and the final conclusion}
Now we consider $S_2$. We split the interval in the following way,
$$
S_2=\sum_{\{R\}}\sum_{r\simeq R} \sum_{\substack{\rho \\ |\gamma| \leqslant T}}\frac{(r+1)^{\rho/c}-r^{\rho/c}}{\rho},
$$
where $\{R\}$ is a certain set of integers satisfying
$$
R \leqslant x^c \quad \text{and} \quad \sum_{\{R\}}1 \ll \log^2x.
$$
Define 
$$
F(s)=\sum_{r\simeq R}r^{-s} \quad \text{and} \quad C(s)=\frac{(1+\frac{1}{R})^{s}-1}{s}.
$$
Hence we arrive at
\begin{align*}
S_2 &= \sum_{\{R\}}\sum_{\substack{\rho \\ |\gamma|\leqslant T}} \sum_{r\simeq R} r^{\rho/c} \frac{\big( (1+\frac{1}{r})^{\rho/c} -1 \big)}{\rho} \\
& \ll \sum_{\{R\}} \bigg| \sum_{\substack{\rho \\ |\gamma|\leqslant T}} \sum_{r\simeq R} r^{\rho/c} \frac{\big( (1+\frac{1}{R})^{\rho/c} -1 \big)}{\rho} \bigg| \\
& \ll \sum_{\{R\}} \bigg| \sum_{\substack{\rho \\ |\gamma|\leqslant T}} F(-\frac{\rho}{c})C(\frac{\rho}{c}) \bigg|.
\end{align*}
By the trivial bound $|C(s)|\ll \frac{1}{R}$, it follows that
\begin{align*}
S_2&\ll \sum_{\{R\}}\frac{1}{R}\Big|\sum_{\substack{\rho \\ |\gamma|\leqslant T}}F(-\frac{\rho}{c})\Big|.
\end{align*}
Using the standard zero spacing process to the sum of $F(-\rho/c)$, namely, we divide the range of the real part of $\rho$ into intervals of length $1/\log x$ to get that
$$
\sum_{\substack{\rho \\ |\gamma|\leqslant T}}\Big|F(-\frac{\rho}{c}) \Big| \ll \log x \sup_{\sigma} \sum_{\substack{\rho \\ |\gamma| \leqslant T , \sigma < \beta \leqslant \sigma + \frac{1}{\log x}}}\Big|F(-\frac{\rho}{c}) \Big|,
$$
where $\beta$ is the real part of $\rho$. Since for $r\simeq R$ and $R \leqslant x^c$, we have for any $\sigma < \beta \leqslant \sigma + 1/\log x$, 
$$
\big| F(-\frac{\rho}{c}) \big| \ll F(-\frac{\sigma}{c}).
$$
Hence by the Cauchy inequality, we obtain that
$$
S_2 \ll \log x\sum_{\{R\}}\frac{1}{R}\bigg| \sup_{\sigma}N(\sigma,T)^{\frac{1}{2}} \bigg(\sum_{\substack{\rho \\ |\gamma|\leqslant T, \sigma < \beta \leqslant \sigma + \frac{1}{\log x}}}\Big(F(-\frac{\sigma}{c}) \Big)^2\bigg)^{\frac{1}{2}}\bigg|.
$$
It follows by Lemma~\ref{lem:mean_value} with $a_n=1$ and $\chi=1$ that
\begin{align*}
S_2&\ll \log^2 x \sum_{\{R\}}\frac{1}{R} \sup_{\sigma} N(\sigma,T)^{\frac{1}{2}} \bigg( (T+R)\sum_{r\simeq R} r^{\frac{2\sigma}{c}}\bigg)^{\frac{1}{2}}\\
&\ll \log^2 x \sum_{\{R\}}  \sup_{\sigma}N(\sigma,T)^{\frac{1}{2}}T^{\frac{1}{2}}R^{-\frac{1}{2}+\frac{\sigma}{c}}.
\end{align*}
By Lemma~\ref{lemma:Maynard} and Lemma~\ref{zero-free}, recalling that $T=x^c\log ^5x$ and $R\leqslant x^c$, we derive that
\begin{align*}
S_2&\ll (\log x)^A \sup_{\sigma\leqslant 1-c(\log t)^{-\frac{2}{3}}(\log\log t)^{-\frac{1}{3}}} N(\sigma,T)^{\frac{1}{2}} x^\sigma \\
&\ll x(\log x)^A \sup_{\sigma\leqslant 1-c(\log t)^{-\frac{2}{3}}(\log\log t)^{-\frac{1}{3}}}x^{-(1-15c/13)(1-\sigma)}\ll x^{1-\eps},
\end{align*}
where $A$ is a constant and the condition $c< 13/15$ ensures that the supremum above contributes a power saving of $x$. Together with \eqref{eq:bdS1} and \eqref{eq:bdS3} we have
\begin{align*}
\sum_{r\leqslant x^c} \bigg( \psi((r+1)^{1/c})-\psi(r^{1/c}) \bigg)=\frac{x}{c\log x} + O(\frac{x}{\log ^2x}) + O(x^{1-\eps}).
\end{align*} 
Eventually by a partial summation, it follows that
$$
\Pi_{c,a}(x) = \frac{x}{c\log ^2x} + O(\frac{x}{\log ^3x})
$$
holds for $0<c< 13/15$.

\section{Proof of Theorem~\ref{thm:main_2}}
Next we turn to the proof of Theorem~\ref{thm:main_2}. Assuming $0<c<1$ it is trivial that
$$
N=m+q \Leftrightarrow \lfloor p^c \rfloor =N-q   \Leftrightarrow (N-q)^{1/c} \leqslant p <(N-q+1)^{1/c}.
$$
Hence the problem changes into a question of the sum of primes in short intervals which is 
\begin{align*}
\Upsilon_c(N) &=\sum_{1\leqslant N-q\leqslant N-2 } \sum_{\substack{ (N-q)^{1/c} \leqslant p <(N-q+1)^{1/c}}}1 \nonumber \\
	&= \sum_{1\leqslant N-q\leqslant N-2 } \big( \pi((N-q+1)^{1/c})-\pi((N-q)^{1/c}) \big)+O(1).
\end{align*}
Following the arguments that are similar to the proof of Theorem~\ref{thm:main_1}, we change $\pi(x)$ into $\psi(x)$ by a partial summation and use Lemma~\ref{lemma:explicit-formula}. Hence the key is to treat the following three sums      
\begin{align*}
	&S_1=\sum_{\substack{1\leqslant N-q\leqslant N-2 }} \bigg( (N-q+1)^{1/c}-(N-q)^{1/c} \bigg),\\
	&S_2=\sum_{\substack{1\leqslant N-q\leqslant N-2}} \sum_{\substack{\rho \\ |\gamma| \leqslant T}}\frac{(N-q+1)^{\rho/c}-(N-q)^{\rho/c}}{\rho},\\
	&S_3= \sum_{\substack{1\leqslant N-q\leqslant N-2}} \frac{(N-q)^{1/c}\log ^2(N-q)}{T}.
\end{align*}  
We note that here the difficulty is to calculate the number of integers $r$ such that $N-r$ is prime if we replace $N-q$ by $r$. Therefore at this point we only aim to show that $S_2+S_3\ll S_1$.
\subsection{The order of $S_1$} 
We give the order of $S_1$ and prove that $S_2+S_3\ll S_1$. We begin with
\begin{equation*}
S_1\gg \frac{1}{\log N}\sum_{2\leqslant n\leqslant N-1}\Lambda(n)(N-n)^{1/c-1}.
\end{equation*}
Thus by a partial summation, we arrive at
\begin{equation}
	\label{S_1}
S_1\gg \frac{1}{\log N}\int_{2}^{N-1}(N-u)^{1/c-1}\,d\big(\sum_{2\leqslant n\leqslant u}\Lambda(n) \big) \gg \frac{N^{1/c}}{\log N}.
\end{equation}

\subsection{The order of $S_2$ and $S_3$}
By taking $T=N\log ^5N$, we immediately find that
\begin{equation}
	\label{S_3}
S_3\ll \frac{1}{\log ^3N}\sum_{\substack{2\leqslant n \leqslant N-1}} N^{1/c-1}\ll \frac{N^{1/c}}{\log ^3N}.
\end{equation}
As for $S_2$, we follow the arguments in section~$3.3$. Also replacing $N-q$ by $r$ and splitting the interval in the following way, it is obtained that
$$
S_2=\sum_{\{R\}}\sum_{\substack{r\simeq R \\ N-r\in\PP}} \sum_{\substack{\rho \\ |\gamma| \leqslant T}}\frac{(r+1)^{\rho/c}-r^{\rho/c}}{\rho},
$$
where $\{R\}$ is a certain set of integers satisfying
$$
R \leqslant x^c \quad \text{and} \quad \sum_{\{R\}}1 \ll \log^2x.
$$
Also define
$$
F(s)=\sum_{r\simeq R}r^{-s} \quad \text{and} \quad C(s)=\frac{(1+\frac{1}{R})^{s}-1}{s}.
$$
Then following the arguments of section~$3.3$, using the standard zero spacing process, the Cauchy-Schwartz inequality, Lemma~\ref{lem:mean_value} by taking $\chi = 1$ and $a_n = 1$, Lemma~\ref{lemma:Maynard} and Lemma~\ref{zero-free} we get that
\begin{equation}
	\label{S_2}
S_2\ll x^{1-\eps}
\end{equation}
holds for $c\in (0, 13/15)$.
Eventually combining \eqref{S_1},\eqref{S_3} and \eqref{S_2}, we achieve that
$$
S_1\gg S_2+S_3.
$$
Thus we have
\begin{equation*}
\Upsilon_c(N) \gg \frac{1}{\log N}\sum_{\substack{1\leqslant N-q\leqslant N-2 }} \bigg( (N-q+1)^{1/c}-(N-q)^{1/c} \bigg)
\end{equation*}
holds for $c\in (0, 13/15)$.

\section{Proof of Theorem~\ref{thm:main_3} and \ref{thm:main_4}}
We first give a proof of Theorem \ref{thm:main_4}, since Theorem~\ref{thm:main_3} has the same argument to Theorem~\ref{thm:main_4} but simpler.
\subsection{The main Proposition}
The following proposition is a key part of the proof.  
\begin{proposition}
\label{prop:main}
For any $C>0$ and $\delta>0$, we have a constant $K=K(C,\delta)$ such that
$$
\Upsilon_c(N)\leqslant K\frac{N^{1/c}}{\log ^2 N},
$$ 
uniformly in $N\geqslant 3$, $0<c\leqslant C$ and $\| c \|\geqslant \delta>0$.
\begin{proof}
For a given $D\geqslant 3$, we have Selberg's sieving weights $\lambda_d$ (see Chapter 3 in \cite{HR}) satisfying
\begin{equation}
	\label{5.1}
\lambda_1=1;\quad |\lambda_d|\leqslant 1;\quad \lambda_d=0\ \text{for}\ d>D;\quad \sum_{d}\sum_{[d_1,d_2]=d}\frac{\lambda_{d_1}\lambda_{d_2}}{[d_1,d_2]}<\frac{1}{\log D}.
\end{equation}
We also have
\begin{equation}
		\label{5.2}
	\bigg( \sum_{d|n}\lambda_d \bigg)^2
	\left\{
	\begin{matrix}
		&\geqslant 0,\ &\text{trivially}, \\
		&=1,\ &\text{if}\ n>D\ \text{is prime}. \\
	\end{matrix}
	\right.
\end{equation}
According to the Selberg's sieve, we know that
\begin{equation}
		\label{5.3}
	\rho_d=\sum_{d=[d_1,d_2]}\lambda_{d_1}\lambda_{d_2}
	\left\{
	\begin{matrix}
		&= 0,\ \text{if}\ d>D^2, \\
		&\ll  3^{\omega(d)}, \\
	\end{matrix}
	\right.
\end{equation}
where $\omega(n)$ denotes the number of different prime factors of $n$. By \eqref{5.2} and \eqref{5.3} we have
\begin{align}
	\label{5.4}
\nonumber \Upsilon_c(N)&=\sum_{p\leqslant N^{1/c}}\sum_{p^c-1<N-q\leqslant p^c}1+O(1)\\
\nonumber&\leqslant \sum_{m\leqslant N^{1/c}}\sum_{m^c-1<N-n\leqslant m^c}\bigg( \sum_{d|m}\lambda_d \bigg)^2\bigg( \sum_{t|n}\lambda_t \bigg)^2\\
\nonumber&\ \ \ +\sum_{\substack{p\leqslant N^{1/c} \\ p<D}}\sum_{p^c-1<N-q\leqslant p^c}1+\sum_{p\leqslant N^{1/c}}\sum_{\substack{p^c-1<N-q\leqslant p^c \\ q<D}}1+O(1)\\
&\leqslant \sum_{m\leqslant N^{1/c}}\sum_{m^c-1<N-n\leqslant m^c}\bigg( \sum_{d|m}\lambda_d \bigg)^2\bigg( \sum_{t|n}\lambda_t \bigg)^2 +D +\frac{D}{c}N^{-1+1/c}.
\end{align}
Here we get the last term since $D$ is a small power of $N$. Next we estimate the first term of \eqref{5.4} by changing the order of sum, which gives that 
$$
\sum_{m\leqslant N^{1/c}}\sum_{m^c-1<N-n\leqslant m^c}\bigg( \sum_{d|m}\lambda_d \bigg)^2\bigg( \sum_{t|n}\lambda_t \bigg)^2 = \sum_{d\leqslant D^2}\sum_{t\leqslant D^2}\rho_d\rho_t \cdot S,
$$
where
\begin{equation*}
	\label{5.5}
S=\sum_{m\leqslant N^{1/c}/d}\sum_{m^cd^c-1<N-nt\leqslant m^cd^c}1
\end{equation*}
is the inner sum. Hence we have
\begin{align*}
S&=\sum_{m\leqslant N^{1/c}/d}\bigg(  \fl{-\frac{N-m^cd^c}{t}} - \fl{-\frac{N-m^cd^c+1}{t}}\bigg)\\
&=\frac{1}{t}\sum_{m\leqslant N^{1/c}/d}1+\sum_{m\leqslant N^{1/c}/d}\bigg( \Psi\Big(-\frac{N-m^cd^c+1}{t}\Big) - \Psi\Big(-\frac{N-m^cd^c}{t}\Big) \bigg).
\end{align*}
Applying Lemma~\ref{lem:Vaaler}, we obtain the second sum of $S$ is $S_1 + O(S_2)$ where
$$
S_1 = \sum_{1\leqslant |h|\leqslant H}a_h\sum_{m\leqslant N^{1/c}/d}\bigg(\e\big(\frac{N-m^cd^c+1}{t}h\big) - \e\big(\frac{N-m^cd^c}{t}h\big)\bigg)
$$
and
$$
S_2 = \sum_{0\leqslant |h|\leqslant H}b_h\sum_{m\leqslant N^{1/c}/d}\bigg(\e\big(\frac{N-m^cd^c+1}{t}h\big) + \e\big(\frac{N-m^cd^c}{t}h\big)\bigg).
$$
We first estimate $S_1$. Since $a_h \ll \frac{1}{h}$ and 
$$
\e(\frac{h}{t})-1 \ll \frac{h}{t},
$$
we have
$$
S_1 \ll  \frac{1}{t}\sum_{1\leqslant |h|\leqslant H} \bigg| \sum_{m\leqslant N^{1/c}/d}\e\big(\frac{m^cd^ch}{t}\big) \bigg|.
$$
To make $S_1$ a error term, we only need to show
$$
\cT \defeq \max_{m' \leqslant N^{1/c}/d}\sum_{1\leqslant |h|\leqslant H}\bigg|\sum_{m\sim m'}\e\big(\frac{m^cd^ch}{t}\big)\bigg| \ll N^{1/c-\varepsilon}.
$$
Thus, by Lemma~\ref{exponential_sum} and summing over $h$ trivially, we obtain that 
\begin{align*}
&\cT\ll N^{\frac{1}{c}+\frac{1-k/c}{2^k-2}+\eps}d^{\frac{k-2^k+2}{2^k-2}}H^{1+\frac{1}{2^k-2}}t^{-\frac{1}{2^k-2}}\\
&\ \ \ +N^{\frac{1}{c}-\frac{4}{c2^k}-\frac{1-k/c}{2^k-2}+\eps}d^{\frac{4}{2^k}-\frac{2^k-2+k}{2^k-2}}H^{1-\frac{1}{2^k-2}}t^{\frac{1}{2^k-2}}.
\end{align*}
Since $D$ and $H$ are both small powers of $x$, thus the integer $k$ satisfies the inequality
\begin{equation}
	\label{5.6}
	\left\{
	\begin{matrix}
		&\frac{1}{c}+\frac{1-k/c}{2^k-2}+\eps<\frac{1}{c}, \\
		&\frac{1}{c}-\frac{4}{c2^k}-\frac{1-k/c}{2^k-2}+\eps<\frac{1}{c}, \\
	\end{matrix}
	\right.
\end{equation}
which concludes that $k=\fl{c+2}$ is the solution of \eqref{5.6}. By taking 
$$
H=D^3\quad \text{and}\quad D=N^{\frac{1}{8(C+1)2^{C}}},
$$
we ensure that $\cT=O(N^{1/c-\eps})$. As for $S_2$, we only need to treat the situation when $h=0$, since the case when $|h|>0$ is the same as $S_1$. The contribution of $S_2$
when $h=0$ is
$$
\ll \frac{N^{1/c}}{dH}\ll N^{1/c-\varepsilon}.
$$
By \eqref{5.4}, we finally achieve that 
\begin{align*}
\Upsilon_c(N)&\leqslant\sum_{d\leqslant D^2}\sum_{t\leqslant D^2}\rho_d\rho_t \frac{N^{1/c}}{dt}+O(\frac{N^{1/c}}{\log ^3N})\\
&\ll \frac{N^{1/c}}{\log ^2D}+O(\frac{N^{1/c}}{\log ^3N}),
\end{align*}
where the implied constant only depends on the constant $\|c\|=\delta$ and $C$. The proof is complete.
\end{proof}
\end{proposition}

\subsection{Proof of Theorem~4}
We consider $\Upsilon_c(N)$ as a function of $c$ and define
$$
F_N(c)=\Upsilon_c(N)\bigg/ \big(\frac{cN^{1/c}}{\log ^2N} \big).
$$
We shall show that $F_N(c)$ is bigger than a constant on average. By Proposition~\ref{prop:main}, we know 
$$
F_N(c)\leqslant K(C,\delta) 
$$
uniformly holds for $0<c\leqslant C$ and $\|c\|>\delta>0$. Also the Lebesgue measure of the natural numbers is negligible, which ensures that $F_N(c)$ is integrable on an interval.

Let $B>A\geqslant \frac{1}{2}$. We start by the definition of $F_N(c)$
$$
\int_{A}^{B}F_N(c)dc = \int_{A}^{B}\frac{\log ^2N}{cN^{1/c}}\bigg(\sum_{p\leqslant N^{1/c}}\sum_{p^c-1<N-q\leqslant p^c}1 +O(1) \bigg)\,dc.
$$
Changing the order of integral and sum, we obtain that
\begin{align}
	\label{FNc1}
\int_{A}^{B}F_N(c)dc = \log^2 N \sum_{p \leqslant N^{1/A}} \sum_{\substack{p^A < N-q \leqslant p^B \\ 2 \leqslant q < N}} \int_{A'}^{B'} \frac{1}{c}N^{-1/c} dc  + O(N^{-1/c+\varepsilon}),
\end{align}
where 
$$
A' = \max(A,\frac{\log(N-q)}{\log p}), \quad B' = \min(B,\frac{\log(N-q+1)}{\log p}).
$$
Since $p^A < N-q \leqslant p^B$, it follows that
\begin{align}
	\label{A'B'}
A' = \frac{\log(N-q)}{\log p} \quad \text{and} \quad B' = \frac{\log(N-q+1)}{\log p}.
\end{align}
Noting that
\begin{align}
	\label{FNc2}
\nonumber \int_{\frac{\log(N-q)}{\log p}}^{\frac{\log(N-q+1)}{\log p}} \frac{1}{c}N^{-1/c} dc &\geqslant \frac{\log p}{\log(N-q+1)} N^{-\frac{\log p}{\log(N-q)}} \cdot ( \frac{\log (N-q+1)}{\log p} - \frac{\log(N-q)}{\log p}) \\
 &\geqslant \frac{N^{-\frac{\log p}{\log(N-q)}}}{\log(N-q+1)}\big(\frac{1}{N-q} - \frac{2}{(N-q)^2}\big),
\end{align}
where we use the inequality
$$
\log (1+x) \geqslant x - \frac{x^2}{2} \quad \text{for all}\ x \geqslant 0.
$$
Inserting \eqref{A'B'} and \eqref{FNc2} into \eqref{FNc1}, we have
\begin{align}
	\label{FNc3}
\nonumber \int_{A}^{B}F_N(c)dc &\geqslant \log^2 N \sum_{p \leqslant N^{1/A}} \sum_{\substack{p^A < N-q \leqslant p^B \\ 2 \leqslant q < N}} \frac{N^{-\frac{\log p}{\log(N-q)}}}{\log(N-q+1)}\big(\frac{1}{N-q} - \frac{2}{(N-q)^2}\big)  \\
\nonumber &+ O(N^{-1/c+\varepsilon}) \\
&\geqslant (1+o(1))\log^2 N \sum_{N^{\varepsilon} \leqslant p \leqslant N^{1/A}} \sum_{\substack{p^A < N-q \leqslant p^B \\ 2 \leqslant q < N}} \frac{N^{-\frac{\log p}{\log(N-q)}}}{(N-q)\log(N-q+1)}.
\end{align}
Recalling that $q$ is a prime, we derive that the inner sum over $q$ is
\begin{align}
	\label{FNc4}
\nonumber&\sum_{\substack{p^A < N-q \leqslant p^B \\ 2 \leqslant q < N}} \frac{N^{-\frac{\log p}{\log(N-q)}}}{(N-q)\log(N-q+1)} = \int_{p^A}^{\min(p^B,N)} \frac{N^{-\frac{\log p}{\log t}}}{t\log (t+1)} d\sum_{\substack{p^A < x N-q \leqslant t \\ q \in \PP}}1 \\
\nonumber&\quad = \int_{p^A}^{\min(p^B,N)} \frac{N^{-\frac{\log p}{\log t}}}{t\log (t+1)} d\Big( \frac{N-t}{\log (N-t)} - \pi(N-t)\Big) \\
&\qquad - \int_{p^A}^{\min(p^B,N)} \frac{N^{-\frac{\log p}{\log t}}}{t\log (t+1)} d\frac{N-t}{\log (N-t)} \defeq I_1 - I_2.
\end{align}
Here we restrict $p>N^{1/B}$ to make $p^B>N$. Thus by prime number theorem, we have $I_1 = o(I_2)$. Hence by \eqref{FNc3} and \eqref{FNc4} it follows that, 
\begin{align}
	\label{FNc5}
\nonumber \int_{A}^{B}F_N(c)dc &\geqslant (1+o(1))\log^2 N \sum_{N^{1/B} < p \leqslant N^{1/A}} \int_{p^A}^{N} \frac{N^{-\frac{\log p}{\log t}}}{t\log (t+1) \log (N-t)}dt \\
&\geqslant (1+o(1)) \sum_{N^{1/B} < p \leqslant N^{1/A}} \int_{p^A}^{N} \frac{N^{-\frac{\log p}{\log t}}}{t}dt.
\end{align}
Next replace $\log t$ by $u$. It follows that
\begin{align*}
\int_{A}^{B}F_N(c)dc &\geqslant (1+o(1))\sum_{ N^{1/B} < p \leqslant N^{1/A}} \int_{A\log p}^{\log N} e^{-\frac{\log p\log N}{u}}du \\
&= (1+o(1)) \sum_{ N^{1/B} < p \leqslant N^{1/A}} \log p \int_{A}^{\frac{\log N}{\log p}} e^{-\frac{\log N}{u}}du.
\end{align*}
Recalling that $p$ is also a prime, by prime number theorem and a partial summation, we obtain that
\begin{align}
	\label{FNc}
\nonumber \int_{A}^{B}F_N(c)dc &\geqslant (1+o(1))\int_{N^{1/B}}^{N^{1/A}} \int_{A}^{\log N / \log t} e^{-\frac{\log N}{u}}du dt \\
&= (1+o(1))\int_{A}^{B} e^{-\frac{\log N}{u}} \int_{N^{1/B}}^{N^{1/u}} dt du = B-A + o(1).
\end{align}

Now we are interested in the following set
$$
S^*=\{ c>0: \lim\limits_{N\rightarrow\infty }\sup F_N(c)<1 \}=\bigcup_{k=0}^{\infty}\bigcup_{T=3}^{\infty}\bigcup_{m=3}^{\infty}\bigcup_{N=3}^{\infty}\bigcap_{n>N}^{\infty}S(J_k,m,n),
$$
where
\begin{equation}
\label{eq:jk}
J_k=[k+\frac{1}{T},k+1-\frac{1}{T}],
\end{equation}
and
$$
S(J_k,m,n)=S=\{ c\in J_k: F_n(c)\leqslant 1-\frac{1}{m} \}.
$$
We shall show that for fixed $m$ and large enough integer $n$ , $\cL(S) = 0$, which implies that $\cL(S^*) = 0$. Suppose that $\cL(S)=l$. For any $\eps>0$ there is a finite cover $I_{\omega}$, $\omega\in\Omega$ and $\Omega$ is an index set such that
$$
S\subseteq \bigcup_{\omega\in\Omega}I_{\omega}\subseteq J_k,\quad \sum_{\omega\in\Omega} \cL(I)<l+\eps.
$$
Now we fix an $I\subseteq J_k$. By \eqref{FNc}, we have for large enough integer $n$ 
\begin{equation}
	\label{lower_bound}
\int_{I}F_n(c)\,dc\geqslant\cL(I)(1-\frac{1}{2m}).
\end{equation}
While from the definition of $S(J_k,m,n)$ and Proposition~\ref{prop:main}, it follows that
\begin{align}
	\label{upper_bound}
\nonumber \int_{I}F_n(c)\, dc &= \int_{S\cap I}F_n(c)\, dc + \int_{\cL(I)-\cL(S\cap I)}F_n(c)\,dc \\
&\leqslant \cL(S\cap I)(1-\frac{1}{m})+(\cL(I)-\cL(S\cap I))K.
\end{align}
Note that the constant $K$ here only depends on $J$. Comparing \eqref{lower_bound} and \eqref{upper_bound}, we get
\begin{align}
	\label{S_I}
\cL(S\cap I)\leqslant \big( 1-\frac{1}{2mK-2m+2} \big)\cL(I).
\end{align}
Summing over $I_{\omega}$ and putting \eqref{S_I} into the sum we have
\begin{align*}
l&=\cL(S)=\cL\big(\bigcup_{\omega\in\Omega}(S\cap I)\big)\leqslant \sum_{\omega\in\Omega}\cL(S\cap I) \\
&\leqslant\big( 1-\frac{1}{2mK-2m+2} \big)\sum_{\omega\in\Omega}\cL(I)\\
&\leqslant \big( 1-\frac{1}{2mK-2m+2} \big)(l+\eps).
\end{align*}
This is impossible except that $l=0$. Here the implied constant $\varepsilon$ depends on $K$ and $m$.

\subsection{A sketch proof of Theorem~\ref{thm:main_3}}
Following the arguments of the proof of Proposition~\ref{prop:main} it can be easily proved for any $C>0$ and $\delta>0$, we have a $K=K(C,\delta)$ such that
$$
\Pi_{c,a}(x)\leqslant K\frac{x}{\log ^2 x},
$$ 
uniformly in $N\geqslant 3$, $0<c\leqslant C$ and $\| c \|\geqslant \delta>0$. Now we consider $\Pi_{c,a}(x)$ as a function of $c$ and define
\begin{equation}
\label{eq:Fup}
F_{x,a}(c)=\Pi_{c,a}(x)\bigg/ \big(\frac{x}{c\log ^2x}\big).
\end{equation}
We shall show that $F_{x,a}(c)$ is exact $1$ in average. 

Let $B> A > 0$. Similar to the last proof, we start with
\begin{align*}
\int_{A}^{B}F_{x,a}(c)dc &=\frac{\log^2x}{x}\int_{A}^{B}\sum_{p\leqslant x}\sum_{p^c-1<q+a\leqslant p^c}c \, dc  \nonumber \\
&=\frac{\log ^2x}{x}\sum_{p\leqslant x}\sum_{p^A<q+a\leqslant p^B}\int_{\frac{\log (q+a)}{\log p}}^{\frac{\log (q+a+1)}{\log p}}c \, dc \\
&=\frac{\log^2x}{x}\sum_{p\leqslant x}\sum_{p^A<q+a\leqslant p^B}\frac{\log q}{q\log ^2p}(1+o(1)).
\end{align*}
By a simple application of prime number theorem and a partial summation, it is obtained that
\begin{equation}
	\label{F_{x,a}(c)_2}
\int_{A}^{B}F_{x,a}(c)dc=(B-A)+o(1),
\end{equation}
where the implied constant depends $A$ and $B$ and tends to $0$ as $x$ tends to $+\infty$.

Similarly, we define that
\begin{equation*}
	S^*=\{ c>0: \lim\limits_{x\rightarrow\infty }\sup F_{x,a}(c)<1 \}=\bigcup_{N=0}^{\infty}\bigcup_{T=3}^{\infty}\bigcup_{m=3}^{\infty}\bigcup_{X=3}^{\infty}\bigcap_{x>X}^{\infty}S(J_N,m,x),
\end{equation*}
where
$$
J_N=[N+\frac{1}{T},N+1-\frac{1}{T}],
$$
and
$$
S(J_N,m,x)=S=\{ c\in J: F_{x,a}(c)\leqslant 1-\frac{1}{m} \}.
$$
In this part, we still aim to prove $\cL(S^*)=0$ to obtain Theorem~3. Suppose that $\cL(S)=l$. For any $\eps>0$, there is a finite cover $I_{\omega}$, $\omega\in\Omega$ and $\Omega$ is index set, such that
$$
S\subseteq \bigcup_{\omega\in\Omega}I_{\omega}\subseteq J,\quad \sum_{\omega\in\Omega} \cL(I)<l+\eps.
$$
Next we fix an $I\subseteq J$. By \eqref{F_{x,a}(c)_2} and \eqref{eq:Fup}, for large enough $x$ we still have an upper and lower bound 
\begin{equation}
	\label{lower_bound'}
	\int_{I}F_{x,a}(c)dc\geqslant \cL(I)(1-\frac{1}{2m})
\end{equation}
and
\begin{equation}
	\label{upper_bound'}
	\int_{I}F_{x,a}(c)\,dc\leqslant \cL(S\cap I)(1-\frac{1}{m})+( \cL(I) - \cL(S\cap I))KN.
\end{equation}
Note that the constant $K$ here only depends on $J$. Comparing \eqref{lower_bound'} and \eqref{upper_bound'}, we get
$$
\cL(S\cap I)\leqslant \big( 1-\frac{1}{2mKN-2m+2} \big) \cL(I),
$$
which is impossible except $l=0$ by summing over $I_{\omega}$. Here the implied constant $\varepsilon$ depends on $K$, $N$ and $m$. This completes the proof.

\section*{Acknowledgement}
This work was supported by the National Natural Science Foundation of China (No. 11901447, 12271422), the Natural Science Foundation of Shaanxi Province (No. 2024JC-YBMS-029). 


\end{document}